\documentclass{amsart}
\usepackage{amsmath,amsthm,amssymb}

\newtheorem{lemma}[equation]{Lemma}
\newtheorem{proposition}[equation]{Proposition}
\newtheorem{corollary}[equation]{Corollary}
\newtheorem{question}[equation]{Question}
\newcommand\bG{{\mathbf G}}
\newcommand\bT{{\mathbf T}}
\newcommand\bs{{\mathbf s}}
\newcommand\BG{{\mathbb G}}
\newcommand\BL{{\mathbb L}}
\newcommand\BT{{\mathbb T}}
\newcommand\BC{{\mathbb C}}
\newcommand\BN{{\mathbb N}}
\newcommand\BZ{{\mathbb Z}}
\newcommand\CH{{\mathcal H}}
\newcommand\CR{{\mathcal R}}
\newcommand\GF{\bG^F}
\DeclareMathOperator\Irr{\mathrm{Irr}}
\DeclareMathOperator\Idch{\mathbf{1}}
\DeclareMathOperator\Id{\mathrm{Id}}
\DeclareMathOperator\refl{\mathrm{refl}}
\DeclareMathOperator\Res{\mathrm{Res}}
\DeclareMathOperator\feg{\mathrm{feg}}
\DeclareMathOperator\Deg{\mathrm{Deg}}
\DeclareMathOperator\GL{\mathrm{GL}}
\DeclareMathOperator\Uch{\mathrm{Uch}}
\DeclareMathOperator\cox{\mathrm{cox}}
\newcommand\inv{^{-1}}
\newcommand\scal[3]{\langle#1,#2\rangle_{#3}}
\newcommand\lexp[2]{\kern\scriptspace\vphantom{#2}^{#1}\kern-\scriptspace#2}
\title[Tower equivalence]{Tower equivalence and Lusztig's truncated Fourier transform}
\author{Jean Michel}
\date{25 July 2022}
\address[J.~Michel]{IMJ-PRG, Universit\'e Paris cit\'e, 
B\^atiment Sophie Germain, 75013, Paris France.}
\email{jean.michel@imj-prg.fr}
\commby{Sarah Witherspoon}
\begin{document}
\begin{abstract}
If  $f$ denotes the  truncated Lusztig Fourier  transform, we show that the
image by $f$ of the normalized characteristic function of a Coxeter element
is   the  alternate   sum  of   the  exterior   powers  of  the  reflection
representation,  and that  any class  function is  tower equivalent  to its
image by $f$. In particular this gives a proof of the results of Chapuy and
Douvropoulos  on  ``Coxeter  factorizations  with  generalized Jucys-Murphy
weights  and matrix tree  theorems for reflection  groups'' for irreducible
spetsial reflection groups, based on Deligne-Lusztig combinatorics.
\end{abstract}
\maketitle
\subsection*{Introduction}
This  paper is a kind  of follow-up to \cite{Mic},  which gives a uniform
proof  of the results of  Chapuy and Stump \cite{CS}  for Weyl groups using
Deligne-Lusztig   combinatorics.  Here  we  extend  this  to  all  spetsial
reflection groups and to the results of Chapuy and Douvropoulos \cite{CD}.

In  the current  paper we  extend  the results  of \cite{Mic}  to all
irreducible   spetsial  groups  (which  coincide  with  the  well-generated
irreducible complex reflection groups except for 6 well-generated primitive
groups  of rank 2 which are  not spetsial, see for example \cite[Corollary
8.3]{Malle3}). In particular, proposition \ref{cchi} (valid for any complex
reflection group) extends \cite[Lemma 3]{Mic} and proposition \ref{Coxeter}
extends \cite[Lemma 5]{Mic}.

Let  us recall some  definitions of Chapuy  and Douvropoulos \cite{CD}. Let
$W\subset\GL(V)$ be a well-generated finite irreducible reflection group on
the  complex vector  space $V$  of dimension  $n$. A  {\em tower}  is a
maximal  chain $\{1\}=W_0\subsetneq W_1\subsetneq\ldots W_n=W$ of parabolic
subgroups  of  $W$;  a  parabolic  subgroup  is  the  fixator of a flat (an
intersection  of some reflecting hyperplanes for  $W$), and is a reflection
subgroup  by a  theorem of  Steinberg. The  chain being  maximal means that
$W_i$ is the fixator of a flat of codimension $i$.

Let $\refl(W)$ be the set of reflections of $W$. To a tower
$T=W_0\subsetneq\ldots\subsetneq  W_n$,  Chapuy  and Douvropoulos \cite{CD}
associate  a set  $J_T=\{J_T^i\}_{i=1,\ldots,n}$ of  elements of  the group
algebra   given  by 
$$J_T^i=\sum_{s\in\refl(W_i)-\refl(W_{i-1})}s,$$
and associate the commutative subalgebra $\BC[J_T]\subset\BC
[W]$.  Two class functions on $W$ are  said to be $T$-equivalent if they have
the  same restriction to  $\BC[J_T]$. Two class  functions $\chi,\chi'$ are
said  to be tower equivalent if they are $T$-equivalent for all towers $T$.
We will simply write $\chi\equiv\chi'$ in this case.

The  main result of \cite{CD} (from which they derive uniformly the others)
is that 
$$\sum_{\chi\in\Irr(W)}\chi(c\inv)\chi\equiv\sum_i (-1)^i\Lambda^i(V)$$
where $c$ is a Coxeter element of $W$ and where on the
right appear the exterior powers of the reflection representation of $W$.

We   derive  this  formula   of  Chapuy-Douvropoulos  from  Deligne-Lusztig
combinatorics, using specifically  Lusztig's truncated Fourier transform.
Let  us recall what is this Fourier transform in the Weyl group case.
Let  $\bG^F$ be a  split finite reductive  group with Frobenius  $F$ and Weyl
group  $W$, and let  $U_\chi$ for $\chi\in\Irr(W)$  be the principal series
unipotent  character indexed by  $\chi$. We define  the ``truncated Lusztig
Fourier  transform'' (truncated  because it  is projected  to the principal
series)  as  the  linear operator which  maps  $\chi$  to  the  class function
$f(\chi)$ on $W$ defined by $$\text{for $w\in W$, }
f(\chi)(w)=\scal{U_\chi}{R_{\bT_w}^\bG(\Idch)}\GF.$$ where
$R_{\bT_w}^\bG(\Idch)$  is the  Deligne-Lusztig character  induced from the
trivial  character of a  torus of type  $w$. The definition  is extended by
linearity to all class functions on $W$.

We extend below this definition to spetsial groups, using the results of
\cite{BMM1}, \cite{BMM} and \cite{Malle1}.
Then the three steps of our proof are
\begin{itemize}
\item For any $w\in W$, we have $\sum_{\chi\in\Irr(W)}\chi(w\inv)\chi\equiv
       \sum_{\chi\in\Irr(W)}\chi(w)\chi$.
\item $f(\sum_{\chi\in\Irr(W)}\chi(c)\chi)=\sum_i(-1)^i\Lambda^i(V)$.
\item For any class function $\chi$ on $W$, we have
$f(\chi)\equiv\chi$.
\end{itemize}

The  third step  implies that  the image  of $\Id-f$ consists of functions
which  are tower equivalent to $0$. A natural question is whether the image
of  $\Id-f$ spans the space of all such functions. It turns out that in all
cases  we  could  compute,  except  surprisingly for the primitive spetsial
group $G_{32}$, these two spaces coincide.

There  are two  features of  any current  work on  spetses which  should be
mentioned.  

The  first  is  that  we use the papers \cite{BMM},  \cite{Malle2} and
\cite{MR}, which  assume the  trace  conjecture  for  cyclotomic  Hecke algebras
(\cite[Theorem-Assumption  2  (1)]{BMM1}).  This  conjecture  is  currently
established  for finite Coxeter groups, and all irreducible spetsial groups
excepted some primitive groups of rank $\ge 3$; see \cite{BCC} for a
recent paper on the topic.

The second is that some properties of spetses are based on the fact that they
satisfy a certain number of ``axioms'' stated in \cite[Chapter 4]{BMM}. Here
the situation is the opposite; these axioms have been checked for
finite Coxeter groups and primitive irreducible spetsial groups, 
but only some of them have been checked in
\cite{Malle} and \cite{Malle1} for imprimitive spetsial groups. We point out
where we use such axioms.

I thank Gunter Malle and the referees for remarks which improved previous
versions of this paper.

\subsection*{First step}
This step, valid for any finite complex reflection group,  results from
\begin{lemma}
For any $\chi\in\Irr(W)$, we have $\chi\equiv\overline\chi$.
\end{lemma}
\begin{proof}
Let $i$ be the anti-involution of $\BC[W]$ induced by $w\mapsto w\inv$ for
$w\in W$. The effect of $i$ on $\chi\in\Irr(W)$
is to send it to $\overline\chi$.
For any tower $T$, the algebra $\BC[J_T]$ is fixed pointwise by $i$ since its
generators are fixed by $i$ and it is a commutative algebra. The lemma follows.
\end{proof}
\subsection*{Coxeter numbers}
In this section $W\subset\GL(V)$ is any finite complex reflection group,
not necessarily irreducible or well-generated.
As  in \cite{GG}, we define the {\em Coxeter number} of $\chi\in\Irr(W)$ as
$\cox_\chi:=\omega_\chi(\CR)$,  the  scalar  by  which  the
central element $\CR=\sum_{s\in\refl(W)}(1-s)$
in  the group algebra acts on the representation underlying $\chi$. 
We denote  by $\feg\chi\in\BN[q]$ the fake degree of
$\chi$, the graded degree of $\chi$ in the coinvariant algebra of $W$, and
define $N(\chi)=(\frac d{dq}\feg\chi)_{q=1}$.

Let $\zeta_e:=\exp(2i\pi/e)\in\BC$. The spetsial Hecke algebra $\CH$ of $W$
(see  \cite[6.4]{BMM})  is  the  Hecke  algebra  over $\BC[q^{\pm 1}]$ with
parameters   given  by   $q,\zeta_e,\zeta_e^2,\ldots,\zeta_e^{e-1}$  for  a
reflection  hyperplane $H$ of $W$ such that $|C_W(H)|=e$. In the following,
we assume that $\CH$ satisfies the trace conjecture since we use results of
\cite{BMM1}, \cite{Malle2} and \cite{MR}..

The   algebra  $\CH$  splits  over   some  extension  $\BC[t^{\pm  1}]$  of
$\BC[q^{\pm 1}]$ where $t^m=q$ for some $m\in\BN$ (one can take $m=|Z(W)|$,
see  \cite[7.2(a)]{Malle2}). The algebra $\BC[W]$ is a deformation of $\CH$
for $t\mapsto 1$, which induces a bijection
$\chi_t\mapsto\chi:\Irr(\CH)\to\Irr(W)$.  Through this bijection the Galois
action  $t\mapsto  e^{2i\pi/m}t$  on  $\Irr(\CH)$  induces a permutation on
$\Irr(W)$ that we will denote $\delta$. The map
$\chi\mapsto\delta(\overline\chi)$   is   an   involution,  called  Opdam's
involution.  We will denote by $S_\chi$ the  Schur element of $\CH$ for the
character  of $\CH$ which specializes to  $\chi$, and we denote by $a_\chi$
and  $A_\chi$, the  valuation and  the degree  in $q$ of $S_{\Idch}/S_\chi$
(they may be rational numbers, not integral, but the next proposition shows
that their sum is integral).

\begin{proposition}\label{cchi}
For any complex reflection group $W$ and any $\chi\in\Irr(W)$ we have
$\cox_\chi=(N(\chi)+N(\overline\chi))/\chi(1)=a_\chi+A_\chi$.
\end{proposition}
\begin{proof}
The  second  equality  is  a  direct  consequence  of  the first one and of
\cite[Corollary 6.9]{BMM1}. Let us prove the first one.

We can reduce to the case where $W$ is irreducible, since both sides of the
first equality are additive for an external tensor product of characters.

As pointed out in \cite[Remark 2]{Mic}, it results from \cite[6.5]{Malle2} that
we   have  $\cox_\chi=(N(\chi)+N(i(\chi)))/\chi(1)$   where  $i$   is  Opdam's
involution. Thus it is sufficient to prove that
$N(\overline\chi)=N(i(\chi))$,  or  equivalently  to  prove  that  for  any
$\chi\in\Irr(W)$  we have $N(\chi)=N(\delta(\chi))$. Let $T_w\in \CH$ be an
element which specializes to $w\in W$. We note that if for
$\chi_t\in\Irr(\CH)$ we have $\chi_t(T_w)\in\BC[q]$ then
$\delta(\chi)(w)=\chi(w)$,  since the Galois  action which defined $\delta$
leaves  $q$ invariant.  We apply  this to  a braid  reflection $\bs$ in the
braid group of $W$ of image $s\in\refl(W)$ in $W$, and whose image in $\CH$
we denote by $T_s$. Then if $s$ is of order $e$, by definition of $\CH$ the
eigenvalues  of $T_s$ in a representation  of character $\chi_t$ are in the
set $\{q,\zeta_e,\ldots,\zeta_e^{e-1}\}$, in particular
$\chi_t(T_s)\in\BC[q]$.  The same considerations apply  to powers of $\bs$,
so  we get that  for any $s\in\refl(W)$  and any $\chi\in\Irr(W)$ we have
$\chi(s)=\delta(\chi)(s)$.  We conclude by  using \cite[formula 1.10]{BMM1}
which states that
$N(\chi)=|\refl(W)|/2-\sum_{s\in\refl(W)}\chi(s)/(1-\overline{\det(s)})$.
\end{proof}

\begin{corollary}\label{aA}
For any complex reflection group, $\cox_\chi$ is
constant on Rouquier families of irreducible characters.
\end{corollary}
\begin{proof}
This is an immediate consequence of the second equality in proposition
\ref{cchi} and \cite[Lemme 2.8]{MR}.
\end{proof}

This is the occasion to give an erratum to \cite{Mic}. The last sentence
of  this paper claims that Corollary  \ref{aA} fails for $G_6$ and $G_8$,
but this was due to a programming error.

\subsection*{Fourier transform}
Let  now $W$ be an irreducible finite spetsial complex reflection group. We
recall  that $W$ is called a spetsial  group if for any $\chi\in\Irr(W)$ we
have  $S_{\Idch}/S_\chi\in\BC[q]$  (it  is  a  priori  only  an  element of
$\BC(t)$).  We have defined in \cite{BMM} for the spetses $\BG=(V,W)$ a set
$\Uch(\BG)$  of ``unipotent characters'', which contains a principal series
$\{U_\chi\}_{\chi\in\Irr(W)}$;
a unipotent characters $\rho$ has a ``generic degree'' $\Deg(\rho)$, 
a polynomial in $q$,
which for the principal series is $\Deg(U_\chi)=S_{\Idch}/S_\chi$.
We also defined ``virtual characters'' $R_{\BT_w}^\BG(\Idch)$ (elements of
$\BZ\Uch(\BG)$),  so the definition we
gave of $f$ by  $f(\chi)(w)=\scal{U_\chi}{R_{\BT_w}^\BG(\Idch)}\BG$
still makes sense for spetsial complex reflection groups.
Since $R_{\BT_w}^\BG(\Idch)$
is a virtual character, $f(\chi)$ takes rational integral values for
$\chi\in\Irr(W)$.

The spetsial complex reflection groups are well-generated, so they have a
unique largest reflection degree $h$, called the Coxeter number. An element
$c\in  W$ is  called a  Coxeter element  if it  is a regular element in the
sense of Springer for the eigenvalue $e^{2i\pi/h}$.
\begin{proposition}\label{Coxeter}
Let $c$ be a Coxeter element of the irreducible spetsial complex reflection
group $W$. For $\chi\in\Irr(W)$ we have 
 $$f(\chi)(c)=\begin{cases}(-1)^i&\text{if $\chi=\Lambda^i(V)$,}\\
                0&\text{otherwise.}
              \end{cases}$$
\end{proposition}
\begin{proof}
For Weyl groups, this is a reformulation of \cite[Lemma 5]{Mic}.

For  the  irreducible  spetsial  primitive  complex  reflection  groups, we
checked  the property  by computer.  Here is  a {\tt  Chevie} \cite{chevie}
program  which returns {\tt true}\ for a  spetsial group $W$ if and only if
Proposition \ref{Coxeter} holds:
\begin{verbatim}
proposition3:=function(W)local un,c,ex;
  un:=UnipotentCharacters(W).harishChandra[1].charNumbers;
  c:=DeligneLusztigCharacter(W,PositionRegularClass(W,
                             Maximum(ReflectionDegrees(W)))).v{un};
  ex:=ChevieCharInfo(W).extRefl;
  return PositionsProperty(c,x->x<>0)=Permuted(ex,SortingPerm(ex))
    and c{ex}=List([0..Length(ex)-1],i->(-1)^i);
end;
\end{verbatim}
It  remains to deal with the  infinite series $G(e,1,n)$ and $G(e,e,n)$. We
use  the $\Phi$-Harish-Chandra theory  in spetsial groups.  Let $\Phi$ be a
$K$-cyclotomic polynomial, where $K$ is the field of definition of $W$, and
 $(\BL,\lambda)$ a $\Phi$-cuspidal pair of the spets $\BG=(V,W)$
(see \cite[4.31]{BMM}). Then we have
the  following decomposition in unipotent  characters indexed by characters
of the relative Weyl group:
$$R_\BL^\BG(\lambda)=\sum_{\varphi\in\Irr(W_\BG(\BL,\lambda))}
  \varepsilon_\varphi\varphi(1)\rho_\varphi$$
where the degree of the unipotent character $\rho_\varphi$ is given by
$$\Deg(\rho_\varphi)=\varepsilon_\varphi \Deg(\lambda)\frac
 {(|\BG|/|\BL|)_{q'}}{S_\varphi}$$
where  $(|\BG|/|\BL|)_{q'}$  is  the  quotient  of  the  polynomial  orders
stripped  of  any factors of $q$ that divide it and  $S_\varphi$  is  the Schur element
associated to $\varphi$ in the relative Hecke algebra $\CH_\BG(\BL,\lambda)$.
If   $\zeta$  is  a   root  of  $\Phi$,   the  relative  Hecke  algebra  is
a $\zeta$-cyclotomic deformation of the group algebra of $W_\BG(\BL,\lambda)$,
which implies that $S_\varphi(\zeta)=|W_\BG(\BL,\lambda)|/\varphi(1)$. Thus
$$\varepsilon_\varphi\varphi(1)=\frac{|W_\bG(\BL,\lambda)|\Deg(\rho_\varphi)
 (\zeta)}{\Deg(\lambda)(\zeta)(|\BG|/|\BL|)_{q'}(\zeta)}.$$
We  want to apply this to the case where $\BL=\BT$ is a Coxeter torus, that
is   when  $\Phi$  is  an  irreducible  factor  of  the  $h$-th  cyclotomic
polynomial and $\zeta=\zeta_h$,  in  which  case  it  simplifies  somewhat:  since $\BL$ is the
centralizer of a $\Phi$-Sylow torus we have by \cite[5.3(ii) and 5.4]{BMM1}
$|W_\BG(\BL,\lambda)|=(|\BG|/|\BL|)_{q'}(\zeta_h)$,
and  since  $\BL$  is  a
torus we have $\Deg(\lambda)=1$ so we get
$\Deg(\rho_\varphi)(\zeta_h)=\varepsilon_\varphi\varphi(1)$. We finally get
$$R_\BT^\BG(\Idch)=\sum_{\varphi\in\Irr(W_\BG(\BT))}
\Deg(\rho_\varphi)(\zeta_h)\rho_\varphi.$$
The  characters $\rho_\varphi$ which appear on  the right-hand side are the
unipotent  characters of the  $\Phi$-principal series ---  they are exactly
the unipotent characters $\rho$ such that $\Deg(\rho)(\zeta_h)\ne 0$.
So  to compute the Fourier transforms $f(\chi)(c)$ it remains to find which
$\rho_\varphi$  are  unipotent  characters  $U_\chi$  in  the principal
series (which are similarly the unipotent characters $\rho$ such that
$\Deg(\rho)(1)\ne 0$), and for these compute $\Deg(U_\chi)(\zeta_h)$.

Malle  describes in \cite[3.14 and 6.10]{Malle} which characters are in the
$\Phi$-principal  series.  For  $G(e,1,n)$  they  are  parameterized by the
$e$-partitions  whose corresponding  symbol has  a $(n,\zeta_e)$-hook (here
$h=en$).  In the principal series, we get the characters $U_{\chi_k}$ where
$\chi_k$ is parameterized by the $e$-partition
$(n-k,1^k,\emptyset,\ldots,\emptyset)$ for $k\in 0,\ldots,n$, which is what
 we expect as $\chi_k=\Lambda^k(V)$.

For  $G(e,e,n)$  characters  are  parameterized by $e$-partitions up to
cyclic  permutations (at least  when the $e$-partition  is not equal to one
its  $e-1$ distinct  cyclic permutations,  otherwise several characters are
parameterized   by  the  same  $e$-partition)  and  similarly  the
$\Phi$-principal series corresponds to
$e$-partitions  whose corresponding symbol has a $(n-1,\zeta_e)$-hook (here
$h=e(n-1)$). For the principal series characters we find exactly the same 
partitions as in the $G(e,1,n)$ case, determining characters $\chi_k$ which
correspond also to the $\Lambda^k(V)$.

It remains to compute $\Deg(U_{\chi_k})(\zeta_h)$. 
Since $U_{\chi_0}=\Idch$ it suffices
to consider $k=1,\ldots,n$. For $G(e,1,n)$ we can use Chlouveraki's
formula \cite[3.2]{CJ} for the value of the corresponding Schur element 
$S_{\chi_k}$ of the Hecke algebra of $G(e,1,n)$ with parameters
$((q,\zeta_e,\ldots,\zeta_e^{e-1}),(q,-1),\ldots,(q,-1))$, related  to the
degree by $\Deg(U_{\chi_k})=S_{\chi_k}\inv\prod_{i=1}^n\frac{q^{ei}-1}{q-1}$,
and after some transformations we get
$$S_{\chi_k}=\frac{e(q^k-1) (q^n-\zeta_e)
\prod_{i=1}^{n-k}(q^{ei}-1)
\prod_{i=1}^{k-1}(q^{ei}-1)}
{q^{k+e\binom k2}(q-1)^n(q^{n-k}-\zeta_e)}$$
whence
$$\Deg(U_{\chi_k})=\frac{q^{k+e\binom k2}(q^{n-k}-\zeta_e)
\prod_{i=k}^n(q^{ei}-1)}
{e(q^k-1) (q^n-\zeta_e) \prod_{i=1}^{n-k}(q^{ei}-1)}$$
To evaluate this at $\zeta_{en}$ we first write
$\frac{q^{en}-1}{q^n-\zeta_e}=\prod_{i\in 0,\ldots,e-1;i\ne 1}(q^n-\zeta_e^i)$
whose value at $\zeta_{en}$ is 
$\prod_{i\in 0,\ldots,e-1;i\ne
1}(\zeta_e-\zeta_e^i)=\zeta_e\inv\prod_{i=1}^{e-1}(1-\zeta_e^i)=e\zeta_e\inv$,
the last equality by taking the value at $q=1$ of $\frac{q^e-1}{q-1}$.
We next evaluate
$$\frac{\prod_{i=k}^{n-1}(q^{ei}-1)}{\prod_{i=1}^{n-k}(q^{ei}-1)}\mid_{q=\zeta_{en}}
=\frac{\prod_{i=k}^{n-1}(\zeta_n^i-1)}{\prod_{i=1}^{n-k}-\zeta_n^i
(\zeta_n^{n-i}-1)}=
\prod_{i=1}^{n-k}(-\zeta_n^{-i})=(-1)^{n-k}\zeta_n^{-\binom{n-k+1}2}$$
and finally get the expected value $(-1)^k$.

For $G(e,e,n)$ we use that the spetsial Hecke algebra $\CH$ is a subalgebra of
the Hecke algebra $\CH'$ of $G(e,1,n)$ with parameters
$((1,\zeta_e,\ldots,\zeta_e^{e-1}),(q,-1),\ldots,(q,-1))$, and the quotient 
$\CH'/\CH$ is of dimension $e$. For a
partition which has no cyclic symmetry (the case of our partitions except
$\chi_1$ for $e=n=2$, which does not occur since $G(2,2,2)$ is not
irreducible) we have $S_\chi=S'_\chi/e$ where
$S'_\chi$ is the Schur element for $\CH'$, since the restriction of $\chi$
to $G(e,e,n)$ is irreducible.
After some transformations we get from the formula of \cite{CJ} for
$S'_\chi$
$$S_{\chi_k}=\frac{a(q^{n-1}-\zeta_e)(q-\zeta_e\inv)(q^{n-k}-1)(q^k-1)
\prod_{i=1}^{n-k-1}(q^{ei}-1)\prod_{i=1}^{k-1}(q^{ei}-1)}
{q^{1+e\binom k2}(q^{n-k-1}\zeta_e\inv-1)(\zeta_e-q^{k-1})(q-1)^n}$$
where the term $(q^k-1)$ is absent if $k=0$ and the term $(q^{n-k}-1)$ absent
if $k=n$, and further, we have $a=1$ if $k$ is $0$ or $n$ and $a=e$ otherwise.
Since $\Deg(U_{\chi_k})=S_{\chi_k}\inv\frac{q^n-1}{q-1}
\prod_{i=1}^{n-1}\frac{q^{ei}-1}{q-1}$ we get
$$\Deg(U_{\chi_k})=
\frac{q^{1+e\binom k2}(q^{n-k-1}\zeta_e\inv-1)(\zeta_e-q^{k-1})(q^n-1)
\prod_{i=k}^{n-1}(q^{ei}-1)}
{a(q^{n-1}-\zeta_e)(q-\zeta_e\inv)(q^{n-k}-1)(q^k-1)
\prod_{i=1}^{n-k-1}(q^{ei}-1)}.$$
To get the value at $\zeta_{e(n-1)}$ similarly to the $G(e,1,n)$ case
we start by simplifying $\frac{q^{e(n-1)}-1}{q^{n-1}-\zeta_e}$, then
$\frac{\prod_{i=k}^{n-2}(q^{ei}-1)}{\prod_{i=1}^{n-k-1}(q^{ei}-1)}
\mid_{q=\zeta_{e(n-1)}}$, and get the expected value $(-1)^k$.
\end{proof}

The  above proof for $G(e,1,n)$ and $G(e,e,n)$ is not really a proof in the
sense  that it has  not yet been  proven that these  groups satisfy all the
axioms  for spetses, including  the unicity of  the Fourier matrix. Another
more direct approach would be to check the result using the values given by
Malle  \cite{Malle} for  the Fourier  matrix, which  is feasible  but takes
several  pages of not very enlightening computations. We have preferred the
approach given above.

The  referee  pointed  out  to  me  the  following facts. A more conceptual
statement  of the above result on  the intersection of the $\Phi$-principal
and  the principal  series is  that the  only $\Phi$-block  of the spetsial
Hecke  algebra  $\CH=\CH_\BG(\BT,\Idch)$  of  defect  1  is  the  principal
$\Phi$-block  (and  all  others  have  defect  0).  This  is  \cite[Theorem
6.6]{BGK} for Weyl groups. For Weyl groups \cite[Theorem 9.6]{Geck} further
shows  that the corresponding Brauer tree  algebra is of type $A$. Assuming
this,  in \cite[page  286]{CM} it  is shown  that the  Schur element ratios
evaluated a $q=\zeta$ are Morita invariants (generalizing
\cite[\S3.7]{Broue}).  Our  result  begs  for  a  conceptual proof of these
observations for spetsial Hecke algebras.

\begin{proposition}\label{symmetric}
For any $w\in W$ we have 
$$
f\left(\sum_{\chi\in\Irr(W)}\chi(w)\chi\right)=
\sum_{\chi\in\Irr(W)}f(\chi)(w)\chi.$$
\end{proposition}
\begin{proof}
The equality is  a  consequence  of  the  fact that the matrix of $f$ is
symmetric, that is for $\chi,\psi\in\Irr{W}$ we have $\scal{f(\chi)}\psi W=
\scal{f(\psi)}\chi W$. For Weyl groups this is a consequence of the explicit
determination  of $f$, see for  example \cite[Theorem 14.2.3]{DM}. There is
also  a  case-free  proof  based  on Shintani descent \cite[III, Corollaire
3.5(iii)]{DM1}.  For spetses it  is a consequence of the axiom that the
Fourier matrix $F_{\chi,\psi}=\scal{U_\chi}
{|W|\inv\sum_w\psi(w)R_{\bT_w}^\bG}\GF$ is symmetric (see \cite{Malle1}, \S 5).
It follows that for $\psi\in\Irr(W)$ we have
$$\begin{aligned}
  \scal{f(\sum_{\chi\in\Irr(W)}\chi(w)\chi)}\psi W&=
\sum_{\chi\in\Irr(W)}\chi(w)\scal{f(\chi)}\psi W\\
  &=\scal{f(\psi)}{\sum_{\chi\in\Irr(W)}\chi(w\inv)\chi}W\\
  &=f(\psi)(w)
\end{aligned}$$
the last equality since the function
$\sum_{\chi\in\Irr(W)}\chi(w\inv)\chi$  is  the  normalized  characteristic
function  of the class of $w$. This shows the proposition.
\end{proof}

It is clear that Proposition \ref{Coxeter} and Proposition \ref{symmetric}
for $w=c$ imply the second step described at the end of the introduction.
\subsection*{Tower equivalence}
We now prove the third step described at the end of the introduction.
\begin{proposition}\label{tower}
For a spetsial reflection group $W$ and a class function
$\chi$ on $W$, we have $\chi\equiv f(\chi)$. 
\end{proposition}
\begin{proof}
We induct on the rank of $W$. Henceforth we assume the statement for every
proper parabolic subgroup $W'$ of $W$.

To  deduce the statement  for $W$ from  the statement for  $W'$, we use the
nice  criterion for tower equivalence given by \cite[Lemma 7.2]{CD}. We say
that  a class  function is  Coxeter-isotypic if  all its  components in the
basis  of  irreducible  characters  have  the  same  Coxeter  number.
Lemma
\cite[7.2]{CD} states that two class functions $\chi$ and $\chi'$ are tower
equivalent  if for any  maximal proper parabolic  subgroup $W'$ of $W$, and
for any Coxeter number, the restrictions of the Coxeter-isotypic components
of $\chi$ and $\chi'$ of same Coxeter number to $W'$ are tower equivalent.

For any irreducible character the class function
$f(\chi)$ is Coxeter-isotypic; this is a consequence of Corollary \ref{aA},
which says that the Coxeter number is constant on Rouquier families,
and the fact that the blocks of the Fourier matrix (the Lusztig families)
coincide with Rouquier families; see \cite{Rou} for Weyl groups and
\cite[\S 5.2]{Malle1} in general.
It is thus sufficient  to prove for any irreducible character $\chi$
and for any maximal proper parabolic subgroup $W'$ of
$W$ that  we have  $\Res^W_{W'}f(\chi)\equiv\Res^W_{W'}\chi$.  
Since by the inductive hypothesis we know that,
if $f'$ is the  truncated Lusztig Fourier transform  on $W'$, we have
$f'(\Res^W_{W'}\chi)\equiv\Res^W_{W'}(\chi)$, it suffices to prove
$$\Res^W_{W'}f(\chi)=f'(\Res^W_{W'}\chi).$$  

Let us thus show that
for $w\in W'$, we have $f(\chi)(w)=f'(\Res^W_{W'}\chi)(w)$, that is, if
$\BG'$ is a split Levi subgroup of $\BG$ of Weyl group $W'$, we have
$\scal{U_\chi}{R^\BG_{\BT_w}(\Idch)}\BG=\scal
 {U_{\Res^W_{W'}\chi}}{R^{\BG'}_{\BT_w}(\Idch)}{\BG'}$. We use the facts:
\begin{itemize}
\item $U_{\Res^W_{W'}\chi}=\lexp *R^\BG_{\BG'} U_\chi$ where $\lexp *R$ is
the Lusztig restriction functor. See for example \cite[Lemma 7.2.11]{DM} for
reductive groups.
This is a general property of commuting algebras --- the algebras involved here
are $\CH_\BG(\BT,\Idch)$ and $\CH_{\BG'}(\BT,\Idch)$---  so it still works for
spetses,  as suggested  in \cite{BMM}  by Axiom  4.6 for $\zeta=1$ or Axiom
4.16(ii) for $(\BL,\lambda)=(\BT,\Idch)$.
\item The transitivity of Deligne-Lusztig induction and the adjunction of
Deligne-Lusztig induction and restriction; see for example \cite[chapter 9]{DM}
for reductive groups. For spetses adjunction is the definition of Lusztig
restriction and for transitivity see for example \cite[Theorem 4.3 (b)]{Malle1}.
\end{itemize}
It follows that
$$\begin{aligned}
\scal{U_{\Res^W_{W'}\chi}}{R^{\BG'}_{\BT_w}(\Idch)}{\BG'}
&=\scal{\lexp *R^\BG_{\BG'}U_\chi}{R^{\BG'}_{\BT_w}(\Idch)}{\BG'}\\
&=\scal{U_\chi}{R_{\BG'}^\BG\circ R^{\BG'}_{\BT_w}(\Idch)}\BG\\
&=\scal{U_\chi}{R^\BG_{\BT_w}(\Idch)}\BG
\end{aligned}$$
the first equality by the first fact, the second by adjunction and the third
by transitivity. This shows the proposition.

\end{proof}

It follows from Proposition \ref{tower} that the image of $\Id-f$ is a vector
space of class functions tower equivalent to $0$. We define the kernel of the
tower equivalence to be the space of all class functions tower equivalent to 0.
Computer calculations suggest the following question:
\begin{question}
Is  it true  that for  any irreducible  spetsial reflection group different
from  $G_{32}$, the image  of $\Id-f$ is  equal to the  kernel of the tower
equivalence?
\end{question}
This question is based on the verification of the equality of the above two
spaces  for many  spetsial reflection  groups, in  particular for  all Weyl
groups  of  rank  $\le  10$,  and  for  all  primitive spetsial groups. The
existence  of one  counter-example suggests  the possibility  of others and
prevents making this question into a conjecture.

For  $G_{32}$ the space of class functions  on $W$ is of dimension 102, the
dimension  of the kernel of tower equivalence  is 78, and that of the image
of  $\Id-f$ is 77.  It is to  note that the  blocks of both matrices on the
basis  of  irreducible  characters  coincide  with the Rouquier families of
characters.  The discrepancy occurs in the 16th family (in the {\tt Chevie}
numbering),  which  contains  7  characters.  In  the space spanned by this
family  the image of $\Id-f$ is of  dimension $5$ while the kernel of tower
equivalence is of dimension 6.
\bibliography{tower}

\providecommand{\bysame}{\leavevmode\hbox to3em{\hrulefill}\thinspace}
\providecommand{\MR}{\relax\ifhmode\unskip\space\fi MR }
\providecommand{\MRhref}[2]{%
  \href{http://www.ams.org/mathscinet-getitem?mr=#1}{#2}
}
\providecommand{\href}[2]{#2}
\begin{thebibliography}{10}

\bibitem{BGK}
F.~Bleher, M.~Geck, and W.~Kimmerle, \emph{Automorphisms of generic
  {I}wahori-{H}ecke algebras and integral group rings of finite {C}oxeter
  groups}, J. Algebra \textbf{197} (1997), 615--655.

\bibitem{BCC}
C.~Boura, E.~Chavli, and M.~Chlouveraki, \emph{The {BMM} symmetrising trace
  conjecture for the exceptional 2-reflection groups of rank 2}, J. Algebra
  \textbf{558} (2020), 176--198.

\bibitem{Broue}
M.~Brou{\'e}, \emph{Equivalences of blocks of group algebras},
  Finite-dimensional algebras and related topics (Ottawa, ON, 1992), Kluwer
  Acad. Publ., Dordrecht, 1994, pp.~1--26.

\bibitem{BMM1}
M.~Brou\'e, G.~Malle, and J.~Michel, \emph{Toward {S}petses {I}},
  Transformation groups \textbf{4} (1999), 157--218.

\bibitem{BMM}
\bysame, \emph{Split spetses for primitive reflection groups}, Astérisque
  \textbf{359} (2014).

\bibitem{CD}
G.~Chapuy and T.~Douvropoulos, \emph{Coxeter factorizations with generalized
  {J}ucys-{M}urphy weights and matrix tree theorems for reflection groups},
  Proc. Lond. Math. Soc. \textbf{126} (2023), 129–191.

\bibitem{CS}
G.~Chapuy and C.~Stump, \emph{Counting factorizations of {C}oxeter elements
  into products of reflections}, J. Lond. Math. Soc. \textbf{90} (2014),
  919--939.

\bibitem{CJ}
M.~Chlouveraki and N.~Jacon, \emph{Schur elements for the {A}riki-{K}oike
  algebra and applications}, J. Algebr. Comb. \textbf{35} (2012), 291--311.

\bibitem{CM}
M.~Chlouveraki and H.~Miyachi, \emph{Decomposition matrices for
  d-{H}arish-{C}handra series: the exceptional rank two cases}, LMS J. Comput.
  Math. \textbf{14} (2011), 271--290.

\bibitem{DM1}
F.~Digne and J.~Michel, \emph{Fonctions {$\mathcal L$} des variétés de
  {D}eligne-{L}usztig et descente de {S}hintani}, Mémoire SMF \textbf{20}
  (1985).

\bibitem{DM}
\bysame, \emph{Representations of finite groups of {L}ie type}, L.M.S. student
  texts, vol.~95, C.U.P., 2020.

\bibitem{Geck}
M.~Geck, \emph{Brauer trees of hecke algebras}, Comm. Algebra \textbf{20}
  (1992), 2937--2973.

\bibitem{GG}
I.~Gordon and S.~Griffeth, \emph{Catalan numbers for complex reflection
  groups}, American J. of Math. \textbf{134} (2012), 1491--1502.

\bibitem{Malle}
G.~Malle, \emph{Unipotente {G}rade imprimitiver komplexer
  {S}piegelungsgruppen}, J. Algebra \textbf{177} (1995), 768--826.

\bibitem{Malle1}
\bysame, \emph{Spetses}, Doc. Math. ICM \textbf{11} (1998), 87–96.

\bibitem{Malle2}
\bysame, \emph{On the rationality and fake degrees of characters of cyclotomic
  {H}ecke algebras}, J. Math. Sci. Univ. Tokyo \textbf{6} (1999), 647--677.

\bibitem{Malle3}
\bysame, \emph{On the generic degrees of cyclotomic algebras}, Representation
  Theory \textbf{4} (2000), 342--369.

\bibitem{MR}
G.~Malle and R.~Rouquier, \emph{Familles de caractères de groupes de
  réflexion complexes}, Representation Theory \textbf{7} (2003), 610--640.

\bibitem{chevie}
J.~Michel, \emph{The development version of the {{\tt Chevie}}\ package of
  {{\tt GAP}}.}, J. Algebra \textbf{435} (2015), 308--336.

\bibitem{Mic}
\bysame, \emph{Deligne-{L}usztig theoretic derivation for {W}eyl groups of the
  number of reflection factorizations of a {C}oxeter element}, Proc. AMS
  \textbf{144} (2016), 937--941.

\bibitem{Rou}
R.~Rouquier, \emph{Familles et blocs d'alg\`ebres de {H}ecke}, C.R.A.S.
  \textbf{329} (1999), 1037--1042.

\end{thebibliography}
\bibliographystyle{amsplain}
\end{document}